\documentclass[reqno]{amsart}
\usepackage{hyperref}

\AtBeginDocument{{\noindent\small \emph{}} \vspace{8mm}}
\begin{document}
\title[\hfil Nonexistence, existence and symmetry]
{Nonexistence, existence and symmetry of normalized ground states to
Choquard equations with a local perturbation}

\author[X. F. Li]
{Xinfu Li}

\address{Xinfu Li \newline
School of Science, Tianjin University of Commerce, Tianjin 300134,
Peoples's Republic of China} \email{lxylxf@tjcu.edu.cn}


\begin{abstract}
We study the Choquard equation with a local perturbation
\begin{equation*}
-\Delta u=\lambda u+(I_\alpha\ast|u|^p)|u|^{p-2}u+\mu|u|^{q-2}u,\
x\in \mathbb{R}^{N}
\end{equation*}
having prescribed mass
\begin{equation*}
\int_{\mathbb{R}^N}|u|^2dx=a^2.
\end{equation*}
For a $L^2$-critical or $L^2$-supercritical perturbation
$\mu|u|^{q-2}u$, we prove nonexistence, existence and symmetry of
normalized ground states, by using the mountain pass lemma, the
Poho\v{z}aev constraint method, the Schwartz symmetrization
rearrangements and some theories of polarizations. In particular,
our results cover the Hardy-Littlewood-Sobolev upper critical
exponent case $p=(N+\alpha)/(N-2)$. Our results are a nonlocal
counterpart of the results in  \cite{{Li 2021-4},{Soave JFA},{Wei-Wu
2021}}.
\end{abstract}

\maketitle \numberwithin{equation}{section}
\newtheorem{theorem}{Theorem}[section]
\newtheorem{lemma}[theorem]{Lemma}
\newtheorem{definition}[theorem]{Definition}
\newtheorem{remark}[theorem]{Remark}
\newtheorem{proposition}[theorem]{Proposition}
\newtheorem{corollary}[theorem]{Corollary}
\allowdisplaybreaks

\section{Introduction and main results}

\setcounter{section}{1}
\setcounter{equation}{0}

We consider the equation
\begin{equation}\label{e1.1}
-\Delta u=\lambda u+(I_\alpha\ast|u|^p)|u|^{p-2}u+\mu|u|^{q-2}u,\
x\in \mathbb{R}^{N},
\end{equation}
where $N\geq 1$, $\alpha\in(0,N)$, $I_{\alpha}$ is the Riesz
potential defined for every $x\in \mathbb{R}^N \setminus \{0\}$ by
\begin{equation}\label{e1.22}
I_{\alpha}(x)=\frac{A_\alpha(N)}{|x|^{N-\alpha}},\
A_\alpha(N)=\frac{\Gamma(\frac{N-\alpha}{2})}{\Gamma(\frac{\alpha}{2})\pi^{N/2}2^\alpha}
\end{equation}
with $\Gamma$ denoting the Gamma function (see \cite{Riesz1949AM},
P.19), $p$ and $q$ will be defined later. The equation (\ref{e1.1})
is usually called the nonlinear Choquard equation.  For the physical
case $N=3$, $p=2$,  $\alpha=2$ and $\mu=0$, (\ref{e1.1}) was
investigated by Pekar in \cite{Pekar 1954} to study the quantum
theory of a polaron at rest. In \cite{Lieb 1977}, Choquard applied
it as an approximation to Hartree-Fock theory of one component
plasma. It also arises in multiple particles systems \cite{Gross
1996} and quantum mechanics \cite{Penrose 1996}.

When looking for solutions to (\ref{e1.1}), a possible choice is to
fix $\lambda<0$ and to search for solutions to (\ref{e1.1}) as
critical points of the action functional
\begin{equation*}
J(u):=\int_{\mathbb{R}^N}\left(\frac{1}{2}|\nabla
u|^2-\frac{\lambda}{2}|u|^2-\frac{1}{2p}(I_\alpha\ast|u|^p)|u|^p-\frac{\mu}{q}|u|^q\right),
\end{equation*}
see for example \cite{{Li-Ma 2020},{Li-Ma-Zhang 2019},{Luo
2020},{Moroz-Schaftingen 2017}} and the references therein.

Alternatively, from a physical point of view, it is interesting to
find solutions of (\ref{e1.1}) having prescribed mass
\begin{equation}\label{e1.3}
\int_{\mathbb{R}^N}|u|^2=a^2.
\end{equation}
In this direction, define on $H^1(\mathbb{R}^N,\mathbb{C})$ the
energy functional
\begin{equation*}
E(u):=\frac{1}{2}\int_{\mathbb{R}^N}|\nabla
u|^2-\frac{1}{2p}\int_{\mathbb{R}^N}(I_\alpha\ast|u|^p)|u|^{p}
-\frac{\mu}{q}\int_{\mathbb{R}^N}|u|^{q}.
\end{equation*}
It is standard to check that $E\in C^1$  under some assumptions on
$p$ and $q$ and a critical point of $E$ constrained to
\begin{equation*}
S_a:=\left\{u\in
H^1(\mathbb{R}^N,\mathbb{C}):\int_{\mathbb{R}^N}|u|^2=a^2\right\}
\end{equation*}
gives rise to a solution to (\ref{e1.1}), satisfying (\ref{e1.3}).
Such solution is usually called a normalized solution of
(\ref{e1.1}). In this method, the parameter $\lambda\in \mathbb{R}$
arises as a Lagrange multiplier, which depends on the solution and
is not a priori given. In this paper, we will focus on the
normalized ground state of (\ref{e1.1}), defined as follows:

\begin{definition}\label{def1.1}
We say that $u$ is a normalized ground state to (\ref{e1.1}) on
$S_a$ if
\begin{equation*}
E(u)=c^{g}:=\inf\left\{E(v):v\in S_a,\ (E|_{S_a})'(v)=0\right\}.
\end{equation*}
The set of the normalized ground states will be denoted by
$\mathcal{G}$.
\end{definition}

When studying normalized solutions of the Choquard equation, three
exponents play an important role: the Hardy-Littlewood-Sobolev upper
critical exponent $\bar{p}$, the Hardy-Littlewood-Sobolev lower
critical exponent $\underline{p}$ and the $L^2$-critical exponent
$p^*$ defined by
\begin{equation*}
\bar{p}:=\left\{\begin{array}{ll}
\infty, & N=1,2,\\
\frac{N+\alpha}{N-2}, & N\geq 3,
\end{array}
\right. \ \  \underline{p}:=\frac{N+\alpha}{N}, \ \
p^*:=1+\frac{2+\alpha}{N}.
\end{equation*}
Recently, researchers pay much attention to the normalized solutions
to the Choquard equation (\ref{e1.1}) for the case $\mu=0$, i.e.,
\begin{equation}\label{e1.12}
-\Delta u=\lambda u+(I_\alpha\ast|u|^p)|u|^{p-2}u,\ x\in
\mathbb{R}^{N}.
\end{equation}
For $\underline{p}<p<p^*$, Ye \cite{Ye 2016} obtained a normalized
ground state to (\ref{e1.12}) by considering the minimizer of $E$
constrained on $S_a$. For $p^*<p<\bar{p}$, the functional $E$ is no
longer bounded from below on $S_a$, Luo \cite{Luo 2019} obtained a
normalized ground state to (\ref{e1.12}) by considering the
minimizer of $E$ constrained on $\mathcal{P}$ defined as in
(\ref{e1.15}). For $p=p^*$, by scaling invariance, the result is
delicate, see \cite{Cazenave-Lions 1982} and \cite{Ye 2016}. In the
case $N\geq 3$,  Li and Ye in \cite{Li-Ye_JMP_2014} considered the
general equation
\begin{equation}\label{e1.123}
-\Delta u=\lambda u+(I_\alpha\ast F(u))f(u),\ x\in \mathbb{R}^{N}
\end{equation}
under a set of assumptions on $f$, which when $f$ takes the special
form $f(s)=C_1|s|^{r-2}s+C_2|s|^{p-2}s$  requires that $p^*<r\leq
p<\bar{p}$. Bartsch, Liu and Liu \cite{Bartsch-Liu-Liu_2020} further
considered the existence of a normalized ground state and the
existence of infinitely many normalized solutions to (\ref{e1.123})
in all dimensions $N\geq 1$. Recently, Yuan, Chen and Tang
\cite{Yuan-Chen-Tang_2020} reconsidered (\ref{e1.123}) with more
general $f\in C(\mathbb{R},\mathbb{R})$. When $p=p^*$ and $2<q<q^*$,
\cite{Liu-Shi 2018} considered the existence and orbital stability
of the normalized ground state to (\ref{e1.1}). Most existing
results considered similar equations to (\ref{e1.1}) with one
positive nonlinearity and one negative nonlinearity, see
\cite{{Bellazzini-Jeanjean-Luo
2013},{Bellazzini-Siciliano},{Cingolani-Jeanjean
2019},{Siciliano-Silva 2020}} for the study of the
Schr\"{o}dinger-Poisson system.

To our knowledge, there are no papers considering the normalized
solutions to the Choquard equation with the Hardy-Littlewood-Sobolev
upper critical exponent $p=\bar{p}$. By \cite{Moroz-Schaftingen JFA
2013}, for fixed $\lambda<0$, (\ref{e1.12}) has no solutions in
$H^1(\mathbb{R}^N)$ under the range $p\geq \bar{p}$. However, the
equation
\begin{equation}\label{e1.13}
-\Delta u=(I_\alpha\ast|u|^{\bar{p}})|u|^{\bar{p}-2}u,\ x\in
\mathbb{R}^N
\end{equation}
has solutions in $D^{1,2}(\mathbb{R}^N)$, see \cite{Gao-Yang-1}. So
it is interesting to study the existence of  normalized solutions to
(\ref{e1.12}) with $p=\bar{p}$ under a local perturbation
$\mu|u|^{q-2}u$, namely equation (\ref{e1.1}). In this paper, we
will give an affirmative answer to the problem.

Now, we  present our  first main result.

\begin{theorem}\label{thm1.1}
Assume $N\geq 1$, $\alpha\in (0,N)$, $a>0$, $\mu>0$, $q^*=2+4/N$,
\begin{equation*}
q^*\leq q<2^*:=\left\{\begin{array}{ll}
\infty, & N=1,2,\\
\frac{2N}{N-2}, & N\geq 3,
\end{array}
\right.  \ \ \ p^*<p\left\{\begin{array}{ll}
<\bar{p}, & N=1,2,\\
\leq \bar{p}, & N\geq 3.
\end{array}
\right.
\end{equation*}
If $q=q^*$, we further assume that $\mu a^{4/N}<(a_N^*)^{4/N}$,
where $a_N^*$ is defined in (\ref{e2.1}). Then the equation
(\ref{e1.1}) has a mountain pass type normalized ground state,
$c^g>0$ if $p<\bar{p}$, and
$$0<c^g<\frac{2+\alpha}{2(N+\alpha)}S_\alpha^{\frac{N+\alpha}{2+\alpha}}$$
if $p=\bar{p}$, where $S_\alpha$ is defined in (\ref{e3.14}).
Moreover, every $u\in \mathcal{G}$ solves (\ref{e1.1}) with some
$\lambda=\lambda(u)<0$.
\end{theorem}

\begin{remark}\label{rmk1.6}
Recently, Yang \cite{Yang 2020} considered the fractional equation
\begin{equation}\label{e1.24}
(-\Delta)^\sigma u=\lambda
u+|u|^{q-2}u+\mu(I_\alpha\ast|u|^p)|u|^{p-2}u,\ x\in \mathbb{R}^{N}
\end{equation}
with $N\geq 2$, $\sigma\in(0,1)$ and $\alpha\in (N-2\sigma,N)$.
Under the assumptions
\begin{equation*}
2+\frac{4\sigma}{N}<q<\frac{2N}{N-2\sigma}\ \ \mathrm{and}\ \
1+\frac{2\sigma+\alpha}{N}\leq p\leq \frac{q}{2}+\frac{\alpha}{N},
\end{equation*}
they obtained a mountain pass type positive radial normalized ground
state to (\ref{e1.24}). Note that the case
$p>\frac{q}{2}+\frac{\alpha}{N}$ was left in \cite{Yang 2020}. In
this paper, by using the Schwartz symmetrization rearrangements, we
can show that $c_r^{mp}=c^{mp}=c^{po}$ and then complement the
interval, see the proof of Theorem \ref{thm1.1}.
\end{remark}

Our second main result is about the positivity and radial symmetry
of the normalized ground states.
\begin{theorem}\label{thm1.2}
Assume  the conditions in Theorem \ref{thm1.1} hold. Let $u$ be a
normalized ground state to (\ref{e1.1}) on $S_a$, then

(1) $|u|>0$ is a normalized ground state to (\ref{e1.1});

(2) there exist $x_0\in \mathbb{R}^N$ and a non-increasing positive
function $v: (0,\infty)\to\mathbb{R}$ such that $|u|=v(|x-x_0|)$ for
almost every $x\in \mathbb{R}^N$;

(3) $u=e^{i\theta}|u|$ for some $\theta\in \mathbb{R}$.
\end{theorem}

By using the methods used in \cite{Li 2021-4}, we can obtain the
following nonexistence result, see \cite{Wei-Wu 2021} for a
different proof.

\begin{theorem}\label{thm1.3}
Let $N\geq 1$, $\alpha\in (0,N)$, $a>0$, $\mu>0$, $q=q^*=2+4/N$,
\begin{equation*}
 p^*<p\left\{\begin{array}{ll}
<\bar{p}, & N=1,2,\\
\leq \bar{p}, & N\geq 3,
\end{array}
\right.
\end{equation*}
and  $\mu a^{4/N}\geq (a_N^*)^{4/N}$ with $a_N^*$ defined in
(\ref{e2.1}). Then $c^{po}=0$ and thus $c^{po}$ can not be attained
and (\ref{e1.1}) has no normalized ground states, where $c^{po}$ is
defined in (\ref{e1.15}).
\end{theorem}

\begin{remark}\label{rmk1.7}
The proof of Theorem \ref{thm1.3} can be done by modifying the proof
of  Theorem 1.5 in \cite{Li 2021-4} done to the Schr\"{o}dinger
equation, so we omit the proof here. The most difference is in Case
2 of Lemma 3.2 in \cite{Li 2021-4} choosing $\tilde{u}$ such that
$\|\tilde{u}\|_2^2=a^2$ and $\|\tilde{u}\|_q^q
=\left(\int_{\mathbb{R}^N}(I_\alpha\ast|\tilde{u}|^p)|\tilde{u}|^p\right)^{\frac{1}{p\eta_p}}$,
where $\eta_p$ is defined in (\ref{e1.25}).
\end{remark}

Now we outline the methods used in this paper to prove Theorems
\ref{thm1.1} and \ref{thm1.2}. For the interaction of
$(I_\alpha\ast|u|^p)|u|^{p-2}u$ and $|u|^{q-2}u$, and the inequality
for the Schwartz symmetrization rearrangement
$$\int_{\mathbb{R}^N}(I_\alpha\ast(|u|^*)^p)(|u|^*)^p\geq \int_{\mathbb{R}^N}(I_\alpha\ast|u|^p)|u|^p,$$
where $|u|^*$ is the Schwartz symmetrization rearrangement of $|u|$,
the methods used in \cite{Bellazzini-Jeanjean-Luo 2013}, \cite{Luo
2019} or \cite{Soave JDE} can not solve our problems for the optimal
range of parameters. In this paper, we combine the methods used in
 \cite{Jeanjean-Le}, \cite{Moroz-Schaftingen
2015}(see also \cite{{Jeanjean-Tanaka 2003},{Li-Ma-Zhang 2019}}) and
\cite{Soave JDE} to prove Theorems \ref{thm1.1} and \ref{thm1.2}.
Using this method, we can treat the existence and symmetry of the
normalized ground states to (\ref{e1.1}) simultaneously. Precisely,
we first use the mountain pass lemma to obtain a Palais-Smale
sequence $\{u_n\}$ of $E$ on $S_{a}\cap H_r^1(\mathbb{R}^N)$ with
$P(u_n)\to 0$ and $E(u_n)\to c_{r}^{mp}$ as $n\to\infty$, where
\begin{equation*}
P(u):=\int_{\mathbb{R}^N}|\nabla
u|^2-\eta_p\int_{\mathbb{R}^N}(I_\alpha\ast|u|^p)|u|^{p}
-\mu\gamma_q\int_{\mathbb{R}^N}|u|^{q}
\end{equation*}
and
\begin{equation}\label{e1.25}
\eta_p:= \frac{N}{2}-\frac{N+\alpha}{2p},\ \
\gamma_q:=\frac{N}{2}-\frac{N}{q}.
\end{equation}
Secondly, by using the Poho\v{z}aev constraint method and the
Schwartz symmetrization rearrangements, we can show that
$c_r^{mp}=c^{mp}=c^{po}$, where
\begin{equation}\label{e1.15}
c^{po}:=\inf_{u\in \mathcal{P}}E(u),\ \ \mathcal{P}:=\left\{u\in
S_a:P(u)=0\right\}.
\end{equation}
Thirdly, by using the radial symmetry of $\{u_n\}$, the bounds of
$c^{po}$, and the relationship  of $c^{po}$ and $c^g$,  we can show
that $\{u_n\}$ converges to a normalized ground state of
(\ref{e1.1}). Lastly, by associating any  normalized ground state
$u$ to a special path in $\Gamma$,  and using $c^{mp}=c^{po}$ and
the theories of polarizations, we can obtain the radial symmetry of
$|u|$.

\medskip

This paper is organized as follows. In Section 2, we cite some
preliminaries. Sections 3 and 4 are devoted to the proof of Theorems
\ref{thm1.1} and \ref{thm1.2}, respectively.

\medskip

\textbf{Notation}: In this paper, it is understood that all
functions, unless otherwise stated, are complex valued, but for
simplicity we write $L^t(\mathbb{R}^N)$, $H^1(\mathbb{R}^N)$,
$D^{1,2}(\mathbb{R}^N)$, .... For $1\leq t<\infty$,
$L^t(\mathbb{R}^N)$ is the usual Lebesgue space endowed with the
norm $\|u\|_t^t:=\int_{\mathbb{R}^N}|u|^t$,  $H^1(\mathbb{R}^N)$ is
the usual Sobolev space endowed with the norm
\begin{equation*}
\|u\|^2:=\int_{\mathbb{R}^N}\left(|\nabla u|^2+|u|^2\right),
\end{equation*}
and $D^{1,2}(\mathbb{R}^N):=\left\{u\in
L^{2^*}(\mathbb{R}^N):|\nabla u|\in L^{2}(\mathbb{R}^N)\right\}$.
$H_{r}^1(\mathbb{R}^N)$ denotes the subspace of functions in
$H^1(\mathbb{R}^N)$ which are radially symmetric with respect to
zero.  $S_{a,r}:=S_a\cap H_{r}^1(\mathbb{R}^N)$. $C,\ C_1,\ C_2,\
...$ denote positive constants, whose values can change from line to
line.

\section{Preliminaries}
\setcounter{section}{2} \setcounter{equation}{0}

The following Gagliardo-Nirenberg inequality can be found in
\cite{Weinstein 1983}.
\begin{lemma}\label{lem2.5}
Let $N\geq 1$ and $2<p<2^*$, then the following sharp
Gagliardo-Nirenberg inequality
\begin{equation*}
\|u\|_{p}\leq C_{N,p}\|u\|_2^{1-\gamma_p}\|\nabla u\|_2^{\gamma_p}
\end{equation*}
holds for any $u\in H^1(\mathbb{R}^N)$, where the sharp constant
$C_{N,p}$ is
\begin{equation*}
C_{N,p}^{p}=\frac{2p}{2N+(2-N)p}\left(\frac{2N+(2-N)p}{N(p-2)}\right)^{\frac{N(p-2)}{4}}\frac{1}{\|Q_{p}\|_2^{p-2}}
\end{equation*}
and $Q_p$ is the unique positive radial solution of equation
\begin{equation*}
-\Delta Q+Q=|Q|^{p-2}Q.
\end{equation*}
In the special case $p=2+4/N$,
$C_{N,p}^{p}=p/\left(2\|Q_{p}\|_2^{4/N}\right)$, or equivalently,
\begin{equation}\label{e2.1}
\|Q_{p}\|_2=\left(\frac{p}{2C_{N,p}^{p}}\right)^{N/4}=:a_N^*.
\end{equation}
\end{lemma}

The following well-known Hardy-Littlewood-Sobolev inequality  can be
found in \cite{Lieb-Loss 2001}.

\begin{lemma}\label{lem HLS}
Let $N\geq 1$, $p$, $r>1$ and $0<\beta<N$ with
$1/p+(N-\beta)/N+1/r=2$. Let $u\in L^p(\mathbb{R}^N)$ and $v\in
L^r(\mathbb{R}^N)$. Then there exists a sharp constant
$C(N,\beta,p)$, independent of $u$ and $v$, such that
\begin{equation*}
\left|\int_{\mathbb{R}^N}\int_{\mathbb{R}^N}\frac{u(x)v(y)}{|x-y|^{N-\beta}}dxdy\right|\leq
C(N,\beta,p)\|u\|_p\|v\|_r.
\end{equation*}
If $p=r=\frac{2N}{N+\beta}$, then
\begin{equation*}
C(N,\beta,p)=C_\beta(N)=\pi^{\frac{N-\beta}{2}}\frac{\Gamma(\frac{\beta}{2})}{\Gamma(\frac{N+\beta}{2})}\left\{\frac{\Gamma(\frac{N}{2})}{\Gamma(N)}\right\}^{-\frac{\beta}{N}}.
\end{equation*}
\end{lemma}

\begin{remark}\label{rek1.31}
(1). By the Hardy-Littlewood-Sobolev inequality above, for any $v\in
L^s(\mathbb{R}^N)$ with $s\in(1,N/\alpha)$, $I_\alpha\ast v\in
L^{\frac{Ns}{N-\alpha s}}(\mathbb{R}^N)$ and
\begin{equation*}
\|I_\alpha\ast v\|_{L^{\frac{Ns}{N-\alpha s}}}\leq C\|v\|_{L^s},
\end{equation*}
where $C>0$ is a constant depending only on $N,\ \alpha$ and $s$.

(2). By the Hardy-Littlewood-Sobolev inequality above and the
Sobolev embedding theorem, we obtain
\begin{equation}\label{e22.4}
\begin{split}
\int_{\mathbb{R}^N}(I_\beta\ast|u|^p)|u|^p\leq
C\left(\int_{\mathbb{R}^N}|u|^{\frac{2Np}{N+\beta}}\right)^{1+\beta/N}
\leq C\|u\|_{H^1(\mathbb{R}^N)}^{2p}
\end{split}
\end{equation}
for any $p\in \left[1+\beta/N,(N+\beta)/(N-2)\right]$ if $N\geq 3$
and $p\in \left[1+\beta/N,+\infty\right)$ if $N=1, 2$, where $C>0$
is a constant depending only on $N,\ \beta$ and $p$.
\end{remark}

The following  Gagliardo-Nirenberg inequality for the convolution
problem can be found in \cite{Feng-Yuan 2015} and
\cite{Moroz-Schaftingen JFA 2013}.

\begin{lemma}\label{lem cGN}
Let $N\geq 1$, $0<\beta<N$, $1+\beta/N<p<\infty$ if $N=1, 2$, and
$1+\beta/N<p<(N+\beta)/(N-2)$ if $N\geq 3$, then
\begin{equation*}
\left(\int_{\mathbb{R}^N}(I_\beta\ast|u|^p)|u|^p\right)^{\frac{1}{2p}}\leq
C_{\beta,p}\|\nabla u\|_2^{\eta_p}\|u\|_2^{1-\eta_p}.
\end{equation*}
The best constant $C_{\beta,p}$ is defined by
\begin{equation*}
(C_{\beta,p})^{2p}=\frac{2p}{2p-Np+N+\beta}\left(\frac{2p-Np+N+\beta}{Np-N-\beta}\right)^{(Np-N-\beta)/2}\|W_p\|_2^{2-2p},
\end{equation*}
where $W_p$ is a radially ground state solution of the elliptic
equation
\begin{equation*}
-\Delta W+W=(I_\beta\ast|W|^p)|W|^{p-2}W.
\end{equation*}
In particular,  in the $L^2$-critical case, i.e., $p=1+(2+\beta)/N$,
$(C_{\beta,p})^{2p}=p\|W_p\|_2^{2-2p}$.
\end{lemma}

\begin{remark}\label{rmk2.3}
Note that $W_p$ may be not unique, but it has the same $L^2$-norm,
see \cite{Feng-Yuan 2015}. Hence, if we define $R_p=\|W_p\|_2$, then
$R_p$ is a constant.
\end{remark}

The following fact is used in this paper (see \cite{{Li-Ma-Zhang
2019},{Moroz-Schaftingen 2015}}). For the readers' convenience, we
give the proof here.

\begin{lemma}\label{lem33.3}
Assume that $N\geq 1$, $\alpha\in (0,N)$, $1+\alpha/N\leq p<\infty$
if $N=1,2$, and $1+\alpha/N\leq p\leq (N+\alpha)/(N-2)$ if $N\geq
3$. Let $\{u_k\}\subset H^1(\mathbb{R}^N)$ be a sequence satisfying
that $u_k\rightharpoonup u$ weakly in $H^1(\mathbb{R}^N)$. Then, for
any $\varphi\in H^1(\mathbb{R}^N)$,
\begin{equation*}
\int_{\mathbb{R}^N}(I_\alpha\ast |u_k|^p)|u_k|^{p-2}u_k\varphi\to
\int_{\mathbb{R}^N}(I_\alpha\ast |u|^p)|u|^{p-2}u\varphi
\end{equation*}
as $k\to \infty$.
\end{lemma}

\begin{proof}
Up to a subsequence, $\{u_k\}$ is bounded in $H^1(\mathbb{R}^N)$,
$u_k\rightharpoonup u$ weakly in $H^1(\mathbb{R}^N)$ and $u_k\to u$
a.e. in $\mathbb{R}^N$.  By the Sobolev embedding theorem, $\{u_k\}$
is bounded in $L^2(\mathbb{R}^N) \cap L^{2^*}(\mathbb{R}^N)$.
Therefore, the sequence $\{|u_k|^p\}$ is bounded in
$L^{\frac{2N}{N+\alpha}}(\mathbb{R}^N)$, and then
\begin{align*}
|u_k|^p & \rightharpoonup  |u|^p \ \text{weakly\ in\ }
L^{\frac{2N}{N+\alpha}}(\mathbb{R}^N).
\end{align*}
By the Rellich theorem, $u_k\to u$ strongly in
$L_{\mathrm{loc}}^r(\mathbb{R}^N)$ for $r\in[1,2^*)$ and then
$|u_k|^{p-2}u_k \rightarrow  |u|^{p-2}u \ \text{strongly\ in\ }
L_{\mathrm{loc}}^{\frac{2Np\delta}{(p-1)(N+\alpha)}}(\mathbb{R}^N)$
with $\delta\in ((N+\alpha)/(2N),1)$ (see Theorem A.2 in
\cite{Willem 1996}). Hence, $|u_k|^{p-2}u_k \varphi\rightarrow
|u|^{p-2}u\varphi $ strongly in
$L^{\frac{2N}{N+\alpha}}(\mathbb{R}^N)$ for any $\varphi\in
C_c^{\infty}(\mathbb{R}^N)$. By Remark \ref{rek1.31}, we have
\begin{equation*}
I_\alpha\ast\left(|u_k|^{p-2}u_k\varphi\right)\to
I_\alpha\ast\left(|u|^{p-2}u\varphi\right)
\end{equation*}
strongly in $L^{\frac{2N}{N-\alpha}}(\mathbb{R}^N)$. Thus,
\begin{equation*}\begin{split}
&\int_{\mathbb{R}^N}(I_\alpha\ast |u_k|^p)|u_k|^{p-2}u_k\varphi-
\int_{\mathbb{R}^N}(I_\alpha\ast |u|^p)|u|^{p-2}u\varphi\\
=&\int_{\mathbb{R}^N}|u_k|^p\left(I_\alpha\ast
(|u_k|^{p-2}u_k\varphi)\right)-
\int_{\mathbb{R}^N} |u|^p\left(I_\alpha\ast(|u|^{p-2}u\varphi)\right)\\
=& \int_{\mathbb{R}^N}|u_k|^p\left(I_\alpha\ast
(|u_k|^{p-2}u_k\varphi)- I_\alpha\ast(|u|^{p-2}u\varphi)\right)\\
&\qquad+
\int_{\mathbb{R}^N}(|u_k|^p-|u|^p)\left(I_\alpha\ast(|u|^{p-2}u\varphi)\right)\\
\to &0
\end{split}
\end{equation*}
as $k\to \infty$. Since $C_c^{\infty}(\mathbb{R}^N)$ is dense in
$H^1(\mathbb{R}^N)$, the proof is complete.
\end{proof}

The following Poho\v{z}aev identity is cited from \cite{Li-Ma 2020},
where the proof is given for $N\geq 3$ and $\lambda>0$ but it
clearly extends to $N=1,2$ and $\lambda\in \mathbb{R}$.
\begin{lemma}\label{lem3.9}
Let $N\geq 1$, $\alpha\in (0,N)$, $\lambda\in \mathbb{R}$, $\mu\in
\mathbb{R}$, $p\in [1+\alpha/N,+\infty)$ and $q\in [2,+\infty)$ for
$N=1,2$, $p\in [1+\alpha/N,(N+\alpha)/(N-2)]$ and $q\in
[2,2N/(N-2)]$ for $N\geq 3$. If $u\in H^1(\mathbb{R}^N)$ is a
solution to (\ref{e1.1}), then $u$ satisfies the Poho\v{z}aev
identity
\begin{align*}
\frac{N-2}{2}\int_{\mathbb{R}^N}|\nabla u|^2
=\frac{N\lambda}{2}\int_{\mathbb{R}^N}|u|^2+\frac{N+\alpha}{2p}\int_{\mathbb{R}^N}(I_{\alpha}\ast|u|^{p})|u|^{p}+\frac{\mu
N}{q}\int_{\mathbb{R}^N}|u|^{q}.
\end{align*}
\end{lemma}

\begin{lemma}\label{lem3.8}
Assume the conditions in Lemma \ref{lem3.9} hold. If $u\in
H^1(\mathbb{R}^N)$ is a solution to (\ref{e1.1}), then $P(u)=0$.
\end{lemma}

\begin{proof}
Multiplying (\ref{e1.1}) by $u$ and integrating over $\mathbb{R}^N$,
we derive
\begin{equation*}
\int_{\mathbb{R}^N}|\nabla
u|^2=\lambda\int_{\mathbb{R}^N}|u|^2+\int_{\mathbb{R}^N}(I_{\alpha}\ast|u|^{p})|u|^{p}+\mu\int_{\mathbb{R}^N}|u|^{q},
\end{equation*}
which combines with the Poho\v{z}aev identity from Lemma
\ref{lem3.9} gives that $P(u)=0$.
\end{proof}

\section{Proof of Theorem \ref{thm1.1}}
\setcounter{section}{3} \setcounter{equation}{0}

In this section, we first study the properties of $c^{po}$ defined
in (\ref{e1.15}), and then give the proof of Theorem \ref{thm1.1}.

For $u\in S_a$ and $s\in \mathbb{R}$, define
\begin{equation}\label{e3.1}
(s\star u)(x):=e^{\frac{N}{2}s}u(e^sx),\ x\in \mathbb{R}^N.
\end{equation}
Then $s\star u\in S_a$. Consider the fiber maps
\begin{equation}\label{e3.4}
\begin{split}
\Psi_u(s):=E(s\star u)&=\frac{1}{2}e^{2s}\int_{\mathbb{R}^N}|\nabla
u|^2-\frac{1}{2p}e^{(Np-N-\alpha)s}\int_{\mathbb{R}^N}(I_\alpha\ast|u|^p)|u|^{p}\\
&\qquad\qquad-\frac{\mu}{q}e^{\left(\frac{N}{2}q-N\right)s}\int_{\mathbb{R}^N}|u|^{q}
\end{split}
\end{equation}
and
\begin{equation*}
\begin{split}
P(s\star u)&=e^{2s}\int_{\mathbb{R}^N}|\nabla
u|^2-\eta_pe^{(Np-N-\alpha)s}\int_{\mathbb{R}^N}(I_\alpha\ast|u|^p)|u|^{p}\\
&\qquad\qquad-\mu\gamma_qe^{\left(\frac{N}{2}q-N\right)s}\int_{\mathbb{R}^N}|u|^{q}.
\end{split}
\end{equation*}
We have the following lemma.

\begin{lemma}\label{lem3.6}
Assume the conditions in Theorem \ref{thm1.1} hold. Then for every
$u\in S_a$, there exists a unique $s_u\in \mathbb{R}$ such that
$P(s_u\star u)=0$. $s_u$ is the unique critical point of the
function $\Psi_u$, and is a strict maximum point at positive level.
Moreover, $P(u)\leq 0$ is equivalent to $s_u\leq 0$.
\end{lemma}

\begin{proof}
Set $P(s\star u)=e^{2s}g_u(s)$, where
\begin{equation*}
\begin{split}
g_u(s)&=\int_{\mathbb{R}^N}|\nabla
u|^2-\eta_pe^{(Np-N-\alpha-2)s}\int_{\mathbb{R}^N}(I_\alpha\ast|u|^p)|u|^{p}\\
&\qquad \qquad
-\mu\gamma_qe^{\left(\frac{N}{2}q-N-2\right)s}\int_{\mathbb{R}^N}|u|^{q}.
\end{split}
\end{equation*}
If $q>q^*$, we have $\frac{N}{2}q-N-2>0$. If $q=q^*$ and $\mu
a^{{4}/{N}}<(a_N^*)^{{4}/{N}}$, we have $\frac{N}{2}q-N-2=0$ and by
the Gagliardo-Nirenberg inequality (Lemma \ref{lem2.5}),
\begin{equation*}
\mu\gamma_q\int_{\mathbb{R}^N}|u|^{q}dx\leq
\mu\gamma_qC_{N,q}^qa^{q(1-\gamma_q)}\|\nabla u\|_2^2<\|\nabla
u\|_2^2.
\end{equation*}
Since $Np-N-\alpha-2>0$, so in both cases, $g_u(s)>0$ for $s\ll0$,
$g_u(s)<0$ for $s\gg0$, and $g_u'(s)<0$ for $s\in \mathbb{R}$. Thus,
$g_u(s)$ has a unique zero $s_u$ as well as $P(s\star u)$. It is
obvious that $P(u)\leq 0\Leftrightarrow s_u\leq 0$.

By direct calculations, we have $\Psi_u'(s)=P(s\star u)$,
$\lim_{s\to -\infty}\Psi_u(s)=0$, $\Psi_u(s)>0$ for $s\ll0$ and
$\lim_{s\to +\infty}\Psi_u(s)=-\infty$. Thus, $s_u$ is the unique
critical point of $\Psi_u(s)$ and $\Psi_u(s_u)=\max_{s\in
\mathbb{R}}\Psi_u(s)>0$.
\end{proof}

\begin{lemma}\label{lem3.1}
Assume the conditions in Theorem \ref{thm1.1} hold.  Then
$c^{po}>0.$
\end{lemma}

\begin{proof}
By Lemma \ref{lem3.6}, $\mathcal{P}\neq\emptyset$.

\textbf{Case 1} ($p\neq \bar{p}$).  For any $u\in \mathcal{P}$, by
the Gagliardo-Nirenberg inequality (Lemmas \ref{lem2.5} and \ref{lem
cGN}), we have
\begin{equation}\label{e3.11}
\begin{split}
\int_{\mathbb{R}^N}|\nabla
u|^2&=\eta_p\int_{\mathbb{R}^N}(I_\alpha\ast|u|^p)|u|^{p}
+\mu \gamma_q\int_{\mathbb{R}^N}|u|^{q}\\
&\leq \eta_p C_{\alpha,p}^{2p}\|u\|_2^{2p(1-\eta_p)}\|\nabla
u\|_2^{2p\eta_p}+\mu
\gamma_qC_{N,q}^q\|u\|_2^{q(1-\gamma_q)}\|\nabla u\|_2^{q\gamma_q}\\
&=\mu \gamma_qC_{N,q}^qa^{q(1-\gamma_q)}\|\nabla
u\|_2^{q\gamma_q}+\eta_p  C_{\alpha,p}^{2p}a^{2p(1-\eta_p)}\|\nabla
u\|_2^{2p\eta_p}.
\end{split}
\end{equation}

If $q>q^*$, then $q\gamma_q>2$. Since $2p\eta_p>2$, (\ref{e3.11})
implies that there exists a constant $C>0$ such that $\|\nabla
u\|_2^2\geq C$. Consequently,
\begin{equation*}
\eta_p\int_{\mathbb{R}^N}(I_\alpha\ast|u|^p)|u|^{p} +\mu
\gamma_q\int_{\mathbb{R}^N}|u|^{q}\geq C.
\end{equation*}

If $q=q^*$ and $\mu a^{{4}/{N}}<(a_N^*)^{{4}/{N}}$, then
 $q\gamma_q=2$ and $\mu
\gamma_qC_{N,q}^qa^{q(1-\gamma_q)}<1$. Since $2p\eta_p>2$,
(\ref{e3.11}) implies that there exists a constant $C>0$ such that
$\|\nabla u\|_2^2\geq C$. Thus, it follows from (\ref{e3.11}) that
\begin{equation*}
\begin{split}
\eta_p\int_{\mathbb{R}^N}(I_\alpha\ast|u|^p)|u|^{p}&\geq \left(1-\mu
\gamma_qC_{N,q}^qa^{q(1-\gamma_q)}\right)\|\nabla u\|_2^2\\
&\geq C\left(1-\mu \gamma_qC_{N,q}^qa^{q(1-\gamma_q)}\right).
\end{split}
\end{equation*}

Any way, there always exists $C_1>0$ such that for any $u\in
\mathcal{P}$,
\begin{equation}\label{e3.2}
E(u)=\left(\frac{\eta_p}{2}-\frac{1}{2p}\right)\int_{\mathbb{R}^N}(I_\alpha\ast|u|^p)|u|^p
+\left(\frac{\gamma_q}{2}-\frac{1}{q}\right)\mu\int_{\mathbb{R}^N}|u|^q\geq
C_1,
\end{equation}
which implies $c^{po}>0$.

\textbf{Case 2} ($p=\bar{p}$). Similarly to Case 1, just in
(\ref{e3.11}), we estimate the term
$\int_{\mathbb{R}^N}(I_\alpha\ast|u|^{\bar{p}})|u|^{\bar{p}}$ by
using (\ref{e3.14}), i.e.,
\begin{equation}\label{e3.16}
\int_{\mathbb{R}^N}(I_\alpha\ast|u|^{\bar{p}})|u|^{\bar{p}}\leq
\left(\frac{\int_{\mathbb{R}^N}|\nabla
u|^2}{S_\alpha}\right)^{\bar{p}}.
\end{equation}
\end{proof}

\begin{lemma}\label{lem3.7}
Assume the conditions in Theorem \ref{thm1.1} hold.  Then there
exists $k>0$ sufficiently small such that
\begin{equation*}0<\sup_{\overline{A_k}}E<c^{po}\ \ \mathrm{and}\ \
u\in \overline{A_k}\Rightarrow E(u),\ P(u)>0,
\end{equation*}
where $A_k=\left\{u\in S_a:\|\nabla u\|_2^2<k\right\}$.
\end{lemma}

\begin{proof}
\textbf{Case 1} ($p\neq \bar{p}$).  By the Gagliardo-Nirenberg
inequality, we have
\begin{equation}\label{e3.15}
\begin{split}
E(u)&\geq \frac{1}{2}\|\nabla u\|_2^2-\frac{\mu}{q}
C_{N,q}^qa^{q(1-\gamma_q)}\|\nabla u\|_2^{q\gamma_q}-\frac{1}{2p}
C_{\alpha,p}^{2p}a^{2p(1-\eta_p)}\|\nabla u\|_2^{2p\eta_p}>0,\\
P(u)&\geq\|\nabla u\|_2^2-\mu
\gamma_qC_{N,q}^qa^{q(1-\gamma_q)}\|\nabla u\|_2^{q\gamma_q}-\eta_p
C_{\alpha,p}^{2p}a^{2p(1-\eta_p)}\|\nabla u\|_2^{2p\eta_p}>0,
\end{split}
\end{equation}
if $u\in \overline{A_k}$ with $k$ small enough, see the proof of
Lemma \ref{lem3.1} for more details. If necessary replacing $k$ with
a smaller quantity, recalling that $c^{po}>0$ by Lemma \ref{lem3.1},
we also have $$E(u)\leq \frac{1}{2}\|\nabla u\|_2^2<c^{po}.$$

\textbf{Case 2} ($p=\bar{p}$). Similarly to Case 1, just in
(\ref{e3.15}), we estimate the term
$\int_{\mathbb{R}^N}(I_\alpha\ast|u|^{\bar{p}})|u|^{\bar{p}}$ by
using (\ref{e3.16}).
\end{proof}

\begin{lemma}\label{lem4.3}
Let $N\geq 3$, $\alpha\in (0,N)$, $a>0$, $\mu>0$, $p=\overline{p}$
and $q^*\leq q<2^*$. If $q=q^*$, we further assume that $\mu
a^{{4}/{N}}<(a_N^*)^{{4}/{N}}$. Then
$$c^{po}<\frac{2+\alpha}{2(N+\alpha)}S_\alpha^{\frac{N+\alpha}{2+\alpha}}.$$
\end{lemma}

\begin{proof}
For any $\epsilon>0$, we define
\begin{equation*}
u_\epsilon(x)=\varphi(x)U_\epsilon(x),
\end{equation*}
where $\varphi(x) \in C_c^{\infty}(\mathbb{R}^N)$ is a cut off
function satisfying: (a) $0\leq \varphi(x)\leq 1$ for any $x\in
\mathbb{R}^N$; (b) $\varphi(x)\equiv 1$ in $B_1$; (c)
$\varphi(x)\equiv 0$ in $\mathbb{R}^N\setminus \overline{B_2}$.
Here, $B_s$ denotes the ball in $\mathbb{R}^N$ of center at origin
and radius $s$.
\begin{equation*}
U_\epsilon(x)=\frac{\left(N(N-2)\epsilon^2\right)^{\frac{N-2}{4}}}{\left(\epsilon^2+|x|^2\right)^{\frac{N-2}{2}}},
\end{equation*}
where $U_1(x)$ is the extremal function of
\begin{equation}\label{e3.14}
\begin{split}
S_\alpha:&=\inf_{ u\in
D^{1,2}(\mathbb{R}^N)\setminus\{0\}}\frac{\int_{\mathbb{R}^N}|\nabla
u|^2}{\left(\int_{\mathbb{R}^N}\left(I_{\alpha}\ast
|u|^{\bar{p}}\right)|u|^{\bar{p}}\right)^{{1}/{\bar{p}}}}.
\end{split}
\end{equation}
In \cite{Gao-Yang-1}, they proved that
$S_\alpha=\frac{S}{(A_\alpha(N) C_\alpha(N))^{{1}/{\bar{p}}}}$,
where $A_\alpha(N)$ is defined in (\ref{e1.22}), $C_\alpha(N)$ is in
Lemma \ref{lem HLS} and
\begin{equation*}
S:=\inf_{ u\in
D^{1,2}(\mathbb{R}^N)\setminus\{0\}}\frac{\int_{\mathbb{R}^N}|\nabla
u|^2}{\left(\int_{\mathbb{R}^N}|u|^{\frac{2N}{N-2}}\right)^{\frac{N-2}{N}}}.
\end{equation*}

By \cite{Brezis-Nirenberg 1983} (see also \cite{Willem 1996}), we
have the following estimates.
\begin{equation*}
\int_{\mathbb{R}^N}|\nabla
u_\epsilon|^2=S^{\frac{N}{2}}+O(\epsilon^{N-2}),\ N\geq 3,
\end{equation*}
and
\begin{equation*}
\int_{\mathbb{R}^N}| u_\epsilon|^2=\left\{\begin{array}{ll}
K_2\epsilon^2+O(\epsilon^{N-2}),& N\geq 5,\\
K_2\epsilon^2|\ln \epsilon|+O(\epsilon^2),& N=4,\\
K_2\epsilon+O(\epsilon^2),& N=3,
\end{array}\right.
\end{equation*}
where $K_2>0$. By direct calculations, for  $t\in (2,2^*)$, there
exists $K_1>0$ such that
\begin{equation*}
\begin{split}
\int_{\mathbb{R}^N}|u_\epsilon|^t &\geq
(N(N-2))^{\frac{N-2}{4}t}\epsilon^{N-\frac{N-2}{2}t}\int_{B_{\frac{1}{\epsilon}}(0)}\frac{1}{(1+|x|^2)^{\frac{N-2}{2}t}}dx\\
&\geq \left\{\begin{array}{ll}
K_1\epsilon^{N-\frac{N-2}{2}t},& (N-2)t>N,\\
K_1\epsilon^{N-\frac{N-2}{2}t}|\ln \epsilon|,& (N-2)t=N,\\
K_1\epsilon^{\frac{N-2}{2}t},& (N-2)t<N.
\end{array}\right.
\end{split}
\end{equation*}
Moreover, similarly as in \cite{Gao-Yang-2} and \cite{Gao-Yang-1},
by direct computations, we have
\begin{equation*}
\int_{\mathbb{R}^N}\left(I_\alpha\ast|u_\epsilon|^{\bar{p}}\right)
|u_\epsilon|^{\bar{p}}\geq (A_\alpha(N)
C_\alpha(N))^{\frac{N}2}S_\alpha^{\frac{N+\alpha}2}
+O(\epsilon^{\frac{N+\alpha}{2}}).
\end{equation*}

Define
$v_\epsilon(x)=(a^{-1}\|u_\epsilon\|_2)^{\frac{N-2}{2}}u_\epsilon(a^{-1}\|u_\epsilon\|_2x)
$. Then
\begin{equation*}
\int_{\mathbb{R}^N}|v_\epsilon|^2=a^2,\ \int_{\mathbb{R}^N}|\nabla
v_\epsilon|^2= \int_{\mathbb{R}^N}|\nabla u_\epsilon|^2,
\end{equation*}
\begin{equation*}
\int_{\mathbb{R}^N}\left(I_\alpha\ast|v_\epsilon|^{\bar{p}}\right)
|v_\epsilon|^{\bar{p}}=\int_{\mathbb{R}^N}\left(I_\alpha\ast|u_\epsilon|^{\bar{p}}\right)
|u_\epsilon|^{\bar{p}},
\end{equation*}
and for $q\in[q^*,2^*)$,
\begin{equation*}
\begin{split}
\int_{\mathbb{R}^N}|v_\epsilon|^{q}&=(a^{-1}\|u_\epsilon\|_2)^{\frac{N-2}{2}q-N}\int_{\mathbb{R}^N}|u_\epsilon|^{q}\\
&\geq a^{N-\frac{N-2}{2}q}\|u_\epsilon\|_2^{\frac{N-2}{2}q-N}K_1\epsilon^{N-\frac{N-2}{2}q}\\
&\geq
\frac{1}{2}a^{N-\frac{N-2}{2}q}K_1K_2^{\frac{N-2}{4}q-\frac{N}{2}}\times\left\{\begin{array}{ll}
1,& N\geq 5,\\
|\ln \epsilon|^{\frac{N-2}{4}q-\frac{N}{2}},& N=4,\\
\epsilon^{\frac{N}{2}-\frac{N-2}{4}q},& N=3.
\end{array}\right.
\end{split}
\end{equation*}
Next we use $v_\epsilon$ to estimate $c^{po}$. By Lemma
\ref{lem3.6}, there exists a unique $s_\epsilon$ such that
$P(s_\epsilon\star v_\epsilon)=0$ and $E(s_\epsilon\star
v_\epsilon)=\max_{s\in \mathbb{R}}E(s\star v_\epsilon)$. Thus,
$c^{po}\leq \max_{s\in \mathbb{R}}E(s\star v_\epsilon)$. By direct
calculations, one has
\begin{equation}\label{e1.117}
\begin{split}
&E(s\star v_\epsilon)\\
&=\frac{1}{2}e^{2s}\int_{\mathbb{R}^N}|\nabla
v_\epsilon|^2-\frac{1}{2\bar{p}}e^{(N\bar{p}-N-\alpha)s}\int_{\mathbb{R}^N}(I_\alpha\ast|v_\epsilon|^{\bar{p}})|v_\epsilon|^{\bar{p}}
-\frac{\mu}{q}e^{(\frac{N}{2}q-N)s}\int_{\mathbb{R}^N}|v_\epsilon|^q\\
&\leq
\frac{1}{2}e^{2s}\left(S^{\frac{N}{2}}+O(\epsilon^{N-2})\right)
-\frac{1}{2\bar{p}}e^{(N\bar{p}-N-\alpha)s}\left((A_\alpha(N)
C_\alpha(N))^{\frac{N}2}S_\alpha^{\frac{N+\alpha}2}
+O(\epsilon^{\frac{N+\alpha}{2}})\right)\\
&\qquad-\frac{\mu}{q}e^{(\frac{N}{2}q-N)s}\frac{1}{2}a^{N-\frac{N-2}{2}q}K_1K_2^{\frac{N-2}{4}q-\frac{N}{2}}\times\left\{\begin{array}{ll}
1,& N\geq 5,\\
|\ln \epsilon|^{\frac{N-2}{4}q-\frac{N}{2}},& N=4,\\
\epsilon^{\frac{N}{2}-\frac{N-2}{4}q},& N=3.
\end{array}\right.
\end{split}
\end{equation}
We claim that there exist $s_0, s_1>0$ independent of $\epsilon$
such that $s_\epsilon\in [-s_0, s_1]$ for $\epsilon>0$ small.
Suppose by contradiction that $s_\epsilon\to -\infty$ or
$s_\epsilon\to +\infty$ as $\epsilon\to 0$. (\ref{e1.117}) implies
that $\max_{s\in \mathbb{R}}E(s\star v_\epsilon)\leq 0$ as
$\epsilon\to 0$ and then $c^{po}\leq 0$, which contradicts Lemma
\ref{lem3.1}. Thus, the claim holds.

In (\ref{e1.117}), $O(\epsilon^{N-2})$ and
$O(\epsilon^{\frac{N+\alpha}{2}})$ can be controlled by the last
term for $\epsilon>0$ small enough. Hence,
\begin{equation*}
\begin{split}
\max_{s\in \mathbb{R}}E(s\star v_\epsilon)&<\sup_{s\in
\mathbb{R}}\left(
\frac{1}{2}e^{2s}S^{\frac{N}{2}}-\frac{1}{2\bar{p}}e^{(N\bar{p}-N-\alpha)s}(A_\alpha(N)
C_\alpha(N))^{\frac{N}2}S_\alpha^{\frac{N+\alpha}2}\right) \\
&\leq\frac{2+\alpha}{2(N+\alpha)}S_\alpha^{\frac{N+\alpha}{2+\alpha}}.
\end{split}
\end{equation*}
The proof is complete.
\end{proof}

\begin{lemma}\label{lem3.3}
Assume the conditions in Theorem \ref{thm1.1} hold.  If  $u\in
\mathcal{P}$ such that $E(u)=c^{po}$, then  $u$ satisfies the
 equation (\ref{e1.1}) with some $\lambda<0$.
\end{lemma}

\begin{proof}
By the Lagrange multipliers rule, there exist $\lambda$ and $\eta$
such that $u$ satisfies
\begin{equation}\label{e3.34}
\begin{split}
-\Delta u-&(I_\alpha\ast|u|^p)|u|^{p-2}u-\mu|u|^{q-2}u\\
&\qquad=\lambda u+\eta[-2\Delta
u-2p\eta_p(I_\alpha\ast|u|^p)|u|^{p-2}u-\mu q\gamma_q|u|^{q-2}u],
\end{split}
\end{equation}
or equivalently,
\begin{equation*}
-(1-2\eta)\Delta u=\lambda u+(1-\eta
2p\eta_p)(I_\alpha\ast|u|^p)|u|^{p-2}u+\mu(1-\eta q\gamma_q)
|u|^{q-2}u.
\end{equation*}

 Next we show $\eta=0$. Similarly to the definition of $P(u)$ (see Lemma \ref{lem3.8}), we obtain
\begin{equation*}
(1-2\eta)\int_{\mathbb{R}^N}|\nabla u|^2-(1-\eta 2p\eta_p)\eta_p
\int_{\mathbb{R}^N}(I_\alpha\ast|u|^p)|u|^{p}-\mu(1-\eta
q\gamma_q)\gamma_q \int_{\mathbb{R}^N}|u|^{q}=0,
\end{equation*}
which combined with $P(u)=0$ gives that
\begin{equation*}
\eta\left(2\int_{\mathbb{R}^N}|\nabla
u|^2-2p\eta_p^2\int_{\mathbb{R}^N}(I_\alpha\ast|u|^p)|u|^{p}-\mu
q\gamma_q^2\int_{\mathbb{R}^N}|u|^{q}\right)=0.
\end{equation*}
If $\eta\neq0$, then
\begin{equation*}
2\int_{\mathbb{R}^N}|\nabla
u|^2-2p\eta_p^2\int_{\mathbb{R}^N}(I_\alpha\ast|u|^p)|u|^{p}-\mu
q\gamma_q^2\int_{\mathbb{R}^N}|u|^{q}=0,
\end{equation*}
which combined with $P(u)=0$ gives that
\begin{equation*}
\begin{cases}
\mu\gamma_q(2p\eta_p-q\gamma_q)\int_{\mathbb{R}^N}|u|^q=(2p\eta_p-2)\int_{\mathbb{R}^N}|\nabla u|^2,   \\
\eta_p(q\gamma_q-2p\eta_p)\int_{\mathbb{R}^N}(I_\alpha\ast|u|^p)|u|^{p}=(q\gamma_q-2)\int_{\mathbb{R}^N}|\nabla
u|^2.
\end{cases}
\end{equation*}
That is a contradiction. So $\eta=0$.

From (\ref{e3.34}) with $\eta=0$, $P(u)=0$, $0< \gamma_q<1$,
$0<\eta_p\leq 1$ and $\mu>0$, we obtain
\begin{equation*}
\begin{split}
\lambda a^2&=\int_{\mathbb{R}^N}|\nabla u|^2-
\int_{\mathbb{R}^N}(I_\alpha\ast|u|^p)|u|^{p}-\mu \int_{\mathbb{R}^N}|u|^{q}\\
&=(\eta_p-1)\int_{\mathbb{R}^N}(I_\alpha\ast|u|^p)|u|^{p}+\mu
(\gamma_q-1) \int_{\mathbb{R}^N}|u|^{q}<0,
\end{split}
\end{equation*}
which implies $\lambda<0$. The proof is complete.
\end{proof}

\begin{lemma}\label{lem4.7}
Assume the conditions in Theorem \ref{thm1.1} hold. If
$\mathcal{C}\neq\emptyset$, then $c^{po}=c^{g}$ and
$\mathcal{C}=\mathcal{G}$, where $\mathcal{C}:=\left\{u\in
\mathcal{P}:E(u)=c^{po}\right\}$.
\end{lemma}

\begin{proof}
For any $u\in \mathcal{C}$, by Lemma \ref{lem3.3}, $u$ is a solution
to (\ref{e1.1}) and thus $E(u)\geq c^{g}$. So $c^{po}\geq c^{g}$
holds. On the other hand, for any normalized solution $v$ of
(\ref{e1.1}) on $S_a$, by Lemma \ref{lem3.8}, $P(v)=0$ and thus
$E(v)\geq c^{po}$, which implies that the reverse inequality
$c^{g}\geq c^{po}$ holds. Hence $c^{g}=c^{po}$ and
$\mathcal{C}=\mathcal{G}$.
\end{proof}

\begin{proposition}\label{pro1.1}
Assume the conditions in Theorem \ref{thm1.1} hold. Let
$\{u_n\}\subset S_{a,r}$ be a Palais-Smale sequence for $E|_{s_a}$
at level $c^{po}$ with $P(u_n)\to 0$ as $n\to \infty$. Then up to a
subsequence, $u_n\to u$ strongly in $H^1(\mathbb{R}^N)$, and $u\in
S_a$ is a radial solution to (\ref{e1.1}) with some $\lambda<0$.
\end{proposition}

\begin{proof}
The proof is divided into four steps.

\textbf{Step 1. We show  $\{u_n\}$ is bounded in
$H^1(\mathbb{R}^N)$.} Since $\{u_n\}\subset S_a$, it is enough to
show that $\{\|\nabla u_n\|_2^2\}$ is bounded.

Case $q>q^*$, it follows from $P(u_n)=o_n(1)$ and
$E(u_n)=c^{po}+o_n(1)$ that
\begin{equation*}
E(u_n)=\left(\frac{\eta_p}{2}-\frac{1}{2p}\right)\int_{\mathbb{R}^N}(I_\alpha\ast|u_n|^p)|u_n|^p
+\left(\frac{\gamma_q}{2}-\frac{1}{q}\right)\mu\int_{\mathbb{R}^N}|u_n|^q+o_n(1),
\end{equation*}
and then
\begin{equation*}
\left(\frac{\eta_p}{2}-\frac{1}{2p}\right)\int_{\mathbb{R}^N}(I_\alpha\ast|u_n|^p)|u_n|^p
+\left(\frac{\gamma_q}{2}-\frac{1}{q}\right)\mu\int_{\mathbb{R}^N}|u_n|^q\leq
C.
\end{equation*}
Using $P(u_n)=o_n(1)$ again yields that
\begin{equation*}
\int_{\mathbb{R}^N}|\nabla
u_n|^2=\eta_p\int_{\mathbb{R}^N}(I_\alpha\ast|u_n|^p)|u_n|^p
+\mu\gamma_q\int_{\mathbb{R}^N}|u_n|^q+o_n(1)\leq C.
\end{equation*}

Case $q=q^*$, from the equality
\begin{equation*}
E(u_n)=\left(\frac{\eta_p}{2}-\frac{1}{2p}\right)\int_{\mathbb{R}^N}(I_\alpha\ast|u_n|^p)|u_n|^p
+o_n(1),
\end{equation*}
we know that $\{\int_{\mathbb{R}^N}(I_\alpha\ast|u_n|^p)|u_n|^p\}$
is bounded. Suppose by contradiction that  $\{\|\nabla u_n\|_2^2\}$
is not bounded, then by $P(u_n)=o_n(1)$, we obtain that
\begin{equation*}
\lim_{n\to \infty}\frac{\mu
\gamma_q\int_{\mathbb{R}^N}|u_n|^q}{\int_{\mathbb{R}^N}|\nabla
u_n|^2}=1,
\end{equation*}
which  contradicts the fact
\begin{equation*}
\mu \gamma_q\int_{\mathbb{R}^N}|u_n|^q\leq\mu
\gamma_qC_{N,q}^qa^{q(1-\gamma_q)}\|\nabla u_n\|_2^{2}<\|\nabla
u_n\|_2^{2}.
\end{equation*}
Hence, $\{\|\nabla u_n\|_2^2\}$ is bounded.

There exists $u\in H_{r}^1(\mathbb{R}^N)$ such that, up to a
subsequence, $u_n\rightharpoonup u$ weakly in $H^1(\mathbb{R}^N)$,
$u_n\to u$ strongly in $L^t(\mathbb{R}^N)$ with $t\in (2,2^*)$ and
$u_n\to u$ a.e. in $\mathbb{R}^N$.

\textbf{Step 2.  We claim that $u\not\equiv0$.} Suppose by
contradiction that $u\equiv 0$.

Case $p<\bar{p}$. By  the assumptions on $p$ and $q$, we have
\begin{equation}\label{e3.25}
\begin{split}
E(u_n)&=\left(\frac{\eta_p}{2}-\frac{1}{2p}\right)\int_{\mathbb{R}^N}(I_\alpha\ast|u_n|^p)|u_n|^p
+\left(\frac{\gamma_q}{2}-\frac{1}{q}\right)\mu\int_{\mathbb{R}^N}|u_n|^q+o_n(1)\\
&=o_n(1),
\end{split}
\end{equation}
which contradicts  $E(u_n)\to c^{po}>0$.

Case $p=\bar{p}$. By using $E(u_n)=c^{po}+o_n(1)$, $P(u_n)=o_n(1)$,
$\int_{\mathbb{R}^N}|u_n|^{q}=o_n(1)$ and (\ref{e3.16}), we get that
\begin{equation*}
E(u_n)=\left(\frac{1}{2}-\frac{1}{2p\eta_p}\right)\int_{\mathbb{R}^N}|\nabla
u_n|^2+o_n(1)
\end{equation*}
and
\begin{equation}\label{e4.4}
\begin{split}
\int_{\mathbb{R}^N}|\nabla u_n|^2&=\eta_{p}\int_{\mathbb{R}^N}(I_\alpha\ast|u_n|^p)|u_n|^{p}+o_n(1)\\
&\leq \eta_{p}\left(\frac{\int_{\mathbb{R}^N}|\nabla
u_n|^2}{S_\alpha}\right)^{p}+o_n(1).
\end{split}
\end{equation}
Since $c^{po}>0$, we obtain
$\liminf_{n\to\infty}\int_{\mathbb{R}^N}|\nabla u_n|^2>0$ and hence
$$\limsup_{n\to \infty}\|\nabla u_n\|_2^2\geq
S_\alpha^{\frac{N+\alpha}{2+\alpha}}.$$ Consequently,
\begin{equation*}
\begin{split}
c^{po}=\lim_{n\to\infty}\left\{\left(\frac{1}{2}-\frac{1}{2p\eta_{p}}\right)\int_{\mathbb{R}^N}|\nabla
u_n|^{2}+o_n(1)\right\}\geq
\frac{2+\alpha}{2(N+\alpha)}S_\alpha^{\frac{N+\alpha}{2+\alpha}},
\end{split}
\end{equation*}
which contradicts Lemma \ref{lem4.3}. So $u\not\equiv0$.

\textbf{Step 3. We show $u$ is a solution to (\ref{e1.1}) with some
$\lambda<0$.} Since $\{u_n\}$ is a Palais-Smale sequence of
$E|_{S_a}$, by the Lagrange multipliers rule, there exists
$\lambda_n$ such that
\begin{equation}\label{e3.10}
\int_{\mathbb{R}^N}\left(\nabla u_n\cdot \nabla \varphi-\lambda_n
u_n\varphi-(I_\alpha\ast|u_n|^p)|u_n|^{p-2}u_n\varphi-\mu|u_n|^{q-2}u_n\varphi\right)=o_n(1)\|\varphi\|
\end{equation}
for every $\varphi\in H^1(\mathbb{R}^N)$.  The choice $\varphi=u_n$
provides
\begin{equation}\label{e3.12}
\lambda_na^2=\int_{\mathbb{R}^N}|\nabla
u_n|^2-\int_{\mathbb{R}^N}(I_\alpha\ast|u_n|^p)|u_n|^{p}-\mu
\int_{\mathbb{R}^N}|u_n|^{q}+o_n(1)
\end{equation}
and the boundedness of $\{u_n\}$ in $H^1(\mathbb{R}^N)$  implies
that $\lambda_n$ is bounded as well; thus, up to a subsequence
$\lambda_n\to \lambda\in \mathbb{R}^N$. Furthermore, by using
$P(u_n)=o_n(1)$, (\ref{e3.12}), $\mu>0$, $\eta_p\in (0,1]$,
$\gamma_q\in(0,1)$ and $u_n\rightharpoonup u$ weakly in
$H^1(\mathbb{R}^N)$, we have
\begin{equation*}
\begin{split}
-\lambda_na^2&=(1-\eta_p)\int_{\mathbb{R}^N}(I_\alpha\ast|u_n|^p)|u_n|^{p}+\mu(1-\gamma_q)
\int_{\mathbb{R}^N}|u_n|^{q}+o_n(1)
\end{split}
\end{equation*}
and then
\begin{equation*}
-\lambda a^2\geq
(1-\eta_p)\int_{\mathbb{R}^N}(I_\alpha\ast|u|^p)|u|^{p}+\mu(1-\gamma_q)
\int_{\mathbb{R}^N}|u|^{q}>0,
\end{equation*}
which implies that $\lambda<0$. By using (\ref{e3.10}) and Lemma
\ref{lem33.3}, we obtain that
\begin{equation}\label{e4.3}
\begin{split}
&\int_{\mathbb{R}^N}\left(\nabla u\cdot \nabla \varphi-\lambda
u\varphi-(I_\alpha\ast|u|^p)|u|^{p-2}u\varphi
-\mu|u|^{q-2}u\varphi\right)\\
&=\lim_{n\to\infty}\int_{\mathbb{R}^N}\left(\nabla u_n\cdot \nabla
\varphi-\lambda_n
u_n\varphi-(I_\alpha\ast|u_n|^p)|u_n|^{p-2}u_n\varphi
-\mu|u_n|^{q-2}u_n\varphi\right)\\
&=\lim_{n\to\infty}o_n(1)\|\varphi\|=0,
\end{split}
\end{equation}
which implies that $u$ satisfies the equation
\begin{equation}\label{e3.13}
-\Delta u=\lambda u+(I_\alpha\ast|u|^p)|u|^{p-2}u+\mu|u|^{q-2}u.
\end{equation}
Thus, $P(u)=0$ by Lemma \ref{lem3.8}.

\textbf{Step 4. We show $u_n\to u$ strongly in $H^1(\mathbb{R}^N)$.}

Case $p<\bar{p}$. Choosing $\varphi=u_n-u$ in (\ref{e3.10}) and
(\ref{e4.3}), and subtracting, we obtain that
\begin{equation*}
\int_{\mathbb{R}^N}(|\nabla (u_n-u)|^2-\lambda|u_n-u|^2)\to 0.
\end{equation*}
Since $\lambda<0$, we have $u_n\to u$ strongly in
$H^1(\mathbb{R}^N)$.

Case $p=\bar{p}$. Set $v_n:=u_n-u$. Then we have
\begin{equation*}
\|u_n\|_2^2=\|u\|_2^2+\|v_n\|_2^2+o_n(1),\ \|\nabla
u_n\|_2^2=\|\nabla u\|_2^2+\|\nabla v_n\|_2^2+o_n(1),
\end{equation*}
\begin{equation*}
\|u_n\|_q^q=\|u\|_q^q+\|v_n\|_q^q+o_n(1)=\|u\|_q^q+o_n(1)
\end{equation*}
and
\begin{equation*}
\int_{\mathbb{R}^N}(I_\alpha\ast
|u_n|^p)|u_n|^p=\int_{\mathbb{R}^N}(I_\alpha\ast
|u|^p)|u|^p+\int_{\mathbb{R}^N}(I_\alpha\ast |v_n|^p)|v_n|^p+o_n(1),
\end{equation*}
which combined with $P(u_n)=o_n(1)$ and $P(u)=0$ gives that
\begin{equation}\label{e4.5}
\int_{\mathbb{R}^N}|\nabla
v_n|^2=\eta_{p}\int_{\mathbb{R}^N}(I_\alpha\ast|v_n|^p)|v_n|^{p}+o_n(1).
\end{equation}
Similarly to (\ref{e4.4}), we infer that
\begin{equation*}
\limsup_{n\to\infty}\|\nabla v_n\|_2^2\geq
S_\alpha^{\frac{N+\alpha}{2+\alpha}}\ \ \mathrm{or}\ \
\liminf_{n\to\infty}\|\nabla v_n\|_2^2=0.
\end{equation*}
If $\limsup_{n\to\infty}\|\nabla v_n\|_2^2\geq
S_\alpha^{\frac{N+\alpha}{2+\alpha}}$, then by $2p\eta_p\geq 2$,
$q\gamma_q\geq 2$, $\mu>0$ and (\ref{e4.5}),
\begin{equation*}
\begin{split}
E(u_n)&=\left(\frac{\eta_p}{2}-\frac{1}{2p}\right)\int_{\mathbb{R}^N}(I_\alpha\ast|u_n|^p)|u_n|^p
+\left(\frac{\gamma_q}{2}-\frac{1}{q}\right)\mu\int_{\mathbb{R}^N}|u_n|^q+o_n(1)\\
&\geq
\left(\frac{\eta_p}{2}-\frac{1}{2p}\right)\int_{\mathbb{R}^N}(I_\alpha\ast|v_n|^p)|v_n|^p+o_n(1)\\
&\geq
\frac{2+\alpha}{2(N+\alpha)}S_\alpha^{\frac{N+\alpha}{2+\alpha}}+o_n(1),
\end{split}
\end{equation*}
which contradicts $E(u_n)=c^{po}+o_n(1)$ and Lemma \ref{lem4.3}.
Thus $\liminf_{n\to\infty}\|\nabla v_n\|_2^2=0$ holds. So up to a
subsequence, $\nabla u_n\to \nabla u$ in $L^2(\mathbb{R}^N)$. Hence,
similarly to the  case  $p<\bar{p}$, we infer that $u_n\to u$
strongly in $H^1(\mathbb{R}^N)$. The proof is complete.
\end{proof}

Now we are ready to give the proof of Theorem \ref{thm1.1}.\\
\textbf{Proof of Theorem \ref{thm1.1}}.  Let $k>0$ be defined by
Lemma \ref{lem3.7}. Following the strategy in  \cite{Jeanjean-Le}
and \cite{Soave JDE} (see also \cite{Jeanjean 1997}), we consider
the augmented functional $\tilde{E}:\mathbb{R}\times
H^1(\mathbb{R}^N)\to \mathbb{R}$
 defined by
\begin{equation}\label{e3.3}
\begin{split}
\tilde{E}(s,u):=E(s\star
u)&=\frac{1}{2}e^{2s}\int_{\mathbb{R}^N}|\nabla
u|^2-\frac{1}{2p}e^{(Np-N-\alpha)s}\int_{\mathbb{R}^N}(I_\alpha\ast|u|^p)|u|^{p}\\
&\qquad-\frac{\mu}{q}e^{\left(\frac{N}{2}q-N\right)s}\int_{\mathbb{R}^N}|u|^{q},
\end{split}
\end{equation}
and consider the restriction $\tilde{E}|_{\mathbb{R}\times
S_{a,r}}$. Note that  $\tilde{E}\in C^1$. Denoting by $E^c$ the
closed sub-level set $\{u \in S_a:E(u)\leq c\}$, we introduce the
mini-max class
\begin{equation}\label{e3.55}
\Gamma_{r}:=\left\{\gamma=(\kappa,\beta)\in C([0,1],\mathbb{R}\times
S_{a,r}):\gamma(0)\in (0,\overline{A_k}), \gamma(1)\in
(0,E^0)\right\}
\end{equation}
with associated mini-max level
\begin{equation*}
c_{r}^{mp}:=\inf_{\gamma\in \Gamma_r}\max_{(s,u)\in
\gamma([0,1])}\tilde{E}(s,u).
\end{equation*}

Let $u\in S_{a,r}$. Since $\int_{\mathbb{R}^N}|\nabla (s\star
u)|^2\to 0^+$ as $s\to -\infty$ and $E(s\star u)\to -\infty$ as
$s\to +\infty$, there exist $s_0\ll1$ and $s_1\gg 1$ such that
\begin{equation}\label{e3.6}
\gamma_u:\tau\in [0,1]\mapsto (0, ((1-\tau)s_0+\tau s_1)\star u)\in
\mathbb{R}\times S_{a,r}
\end{equation}
is a path in $\Gamma_r$. The continuity of $\gamma_u$ follows from
the fact that
\begin{equation}\label{e3.17}
(s,u)\in  \mathbb{R}\times H^1(\mathbb{R}^N) \mapsto (s\star u) \in
H^1(\mathbb{R}^N)\ \mathrm{ is\  continuous},
\end{equation}
see Lemma 3.5 in  \cite{Bartsch-Soave 2019}. Hence $c_r^{mp}$ is
well defined.

To study the value of $c_r^{mp}$, we also consider the mini-max
level
\begin{equation}\label{e3.21}
c^{mp}:=\inf_{\gamma\in \Gamma}\max_{(s,u)\in
\gamma([0,1])}\tilde{E}(s,u)
\end{equation}
with
\begin{equation}\label{e3.5}
\Gamma:=\left\{\gamma=(\kappa,\beta)\in C([0,1],\mathbb{R}\times
S_a):\gamma(0)\in (0,\overline{A_k}), \gamma(1)\in (0,E^0)\right\}.
\end{equation}
Obviously,  $c_r^{mp}\geq c^{mp}$.

For any $\gamma=(\kappa,\beta)\in \Gamma$, consider the function
\begin{equation*}
P_\gamma:\tau\in [0,1]\mapsto P(\kappa(\tau)\star \beta(\tau))\in
\mathbb{R}.
\end{equation*}
We have $P_\gamma(0)=P(\beta(0))>0$ by Lemma \ref{lem3.7}, and by
Lemma \ref{lem3.6}, $P_\gamma(1)=P(\beta(1))<0$ since
$E(\beta(1))\leq 0$. Moreover, $P_\gamma$ is continuous by
(\ref{e3.17}), and hence there exists $\tau_\gamma\in (0, 1)$ such
that $P_\gamma(\tau_\gamma)=0$, namely
$\kappa(\tau_\gamma)\star\beta(\tau_\gamma)\in \mathcal{P}$; this
implies that
\begin{equation}\label{e3.22}
\max_{\gamma([0,1])}\tilde{E}\geq
\tilde{E}(\gamma(\tau_\gamma))=E(\kappa(\tau_\gamma)\star\beta(\tau_\gamma))\geq
\inf_{\mathcal{P}}E=c^{po},
\end{equation}
and consequently $c^{mp}\geq c^{po}$.

For any $u\in \mathcal{P}$, let $|u|^*$ be the Schwartz
symmetrization rearrangement of $|u|$. Since $\||u|^*\|_t=\|u\|_t$
with $t\in[1,\infty)$, $\|\nabla (|u|^*)\|_2\leq \|\nabla u\|_2$ and
\begin{equation*}
\int_{\mathbb{R}^N}(I_\alpha\ast(|u|^*)^p)(|u|^*)^p\geq
\int_{\mathbb{R}^N}(I_\alpha\ast|u|^p)|u|^p,
\end{equation*}
we obtain that $\Psi_{|u|^*}(s)\leq \Psi_u(s)$ for any $s\in
\mathbb{R}$, where $\Psi_u(s)$ is defined in (\ref{e3.4}). Let $s_u$
be defined by Lemma \ref{lem3.6}  be such that $P(s_u\star u)=0$.
Then
\begin{equation*}
E(u)=\Psi_u(0)=\Psi_u(s_u)\geq \Psi_u(s_{|u|^*})\geq
\Psi_{|u|^*}(s_{|u|^*}).
\end{equation*}
Since $s_{|u|^*}\star |u|^*\in \mathcal{P}\cap H_r^1(\mathbb{R}^N)$,
we have that
\begin{equation*}
E(u)\geq \inf_{\mathcal{P}\cap H_r^1(\mathbb{R}^N)}E(u),
\end{equation*}
which implies that
\begin{equation}\label{e3.23}
c^{po}=\inf_{\mathcal{P}}E(u)\geq \inf_{\mathcal{P}\cap
H_r^1(\mathbb{R}^N)}E(u).
\end{equation}

For any $u\in \mathcal{P}\cap H_r^1(\mathbb{R}^N)$,  $\gamma_u$
defined in (\ref{e3.6}) is a path in $\Gamma_r$ with
\begin{equation}\label{e3.24}
E(u)=\max_{\gamma_u([0,1])}\tilde{E}\geq c_r^{mp},
\end{equation}
which implies
\begin{equation*}
\inf_{\mathcal{P}\cap H_r^1(\mathbb{R}^N)}E(u)\geq c_r^{mp}.
\end{equation*}

Now, we have proved that
\begin{equation*}
c^{po}=c^{mp}=c_r^{mp}>\sup_{(\overline{A_k}\cup E^0)\cap
S_{a,r}}E=\sup_{((0,\overline{A_k})\cup (0,E^0))\cap
(\mathbb{R}\times S_{a,r})}\tilde{E}.
\end{equation*}
Using the terminology of Section 5 in \cite{Ghoussoub 93}, this
means that $\{\gamma([0, 1]) :\gamma\in \Gamma_r\}$ is a homotopy
stable family of compact subsets of $\mathbb{R}\times S_{a,r}$ with
extended closed boundary $(0, \overline{A_k})\cup(0, E^0)$, and that
the super-level set $\{\tilde{E}\geq c^{po}\}$ is a dual set for
$\Gamma_r$, in the sense that assumptions ($F'1$) and ($F'2$) in
Theorem 5.2 of \cite{Ghoussoub 93} are satisfied. Therefore, taking
any minimizing sequence $\{\gamma_n=(\kappa_n, \beta_n)\}\subset
\Gamma_r$ for $c^{po}$ with the property that $\kappa_n\equiv0$ and
 $\beta_n(\tau)\geq 0$ a.e. in $\mathbb{R}^N$ for every $\tau\in[0, 1]$, there exists a Palais-Smale sequence
$\{(s_n, w_n)\}\subset \mathbb{R}\times S_{a,r}$ for
$\tilde{E}|_{\mathbb{R}\times S_{a,r}}$ at level $c^{po}$, that is,
$\tilde{E}(s_n, w_n)\to c^{po}$,
\begin{equation}\label{e3.7}
\partial_s\tilde{E}(s_n, w_n)\to 0\ \ \mathrm{and}\ \ \|\partial_u\tilde{E}(s_n,
w_n)\|_{({T_{w_n}S_{a,r}})^*}\to 0\ \ \mathrm{as}\ n\to\infty,
\end{equation}
with the additional property that
\begin{equation}\label{e3.8}
|s_n|+\mathrm{dist}_{H^1(\mathbb{R}^N)}(w_n,\beta_n([0,1]))\to 0 \
\mathrm{as}\ n\to\infty.
\end{equation}
By (\ref{e3.7}), the first condition in (\ref{e3.7}) reads
$P(s_n\star w_n)\to 0$, while the second condition gives that
\begin{equation*}
\begin{split}
e^{2s_n}\int_{\mathbb{R}^N}\nabla w_n\cdot \nabla \phi
-&e^{(Np-N-\alpha)s_n}\int_{\mathbb{R}^N}(I_\alpha\ast|w_n|^p)|w_n|^{p-2}w_n\phi\\
&-\mu
e^{(\frac{N}{2}q-N)s_n}\int_{\mathbb{R}^N}|w_n|^{q-2}w_n\phi=o_n(1)\|\phi\|
\end{split}
\end{equation*}
for every $\phi\in T_{w_n}S_{a,r}$. Since $\{s_n\}$ is bounded due
to (\ref{e3.8}), this is equivalent to
\begin{equation}\label{e3.9}
dE(s_n\star w_n)[s_n\star
\phi]=o_n(1)\|\phi\|=o_n(1)\|s_n\star\phi\| \ \ \mathrm{as}\
n\to\infty.
\end{equation}
Let $u_n:=s_n\star w_n$. By Lemma 5.8 in \cite{Soave JDE}, equation
(\ref{e3.9}) establishes that $\{u_n\}\subset S_{a,r}$ is a
Palais-Smale sequence for $E|_{S_{a,r}}$ (thus a Palais-Smale
sequence for $E|_{S_{a}}$, since the problem is invariant under
rotations) at level $c^{po}$, with $P(u_n)\to 0$ as $n\to\infty$.

By Proposition \ref{pro1.1}, up to a subsequence, $u_n\to u$
strongly in $H^1(\mathbb{R}^N)$. Thus, $u\in S_a$ is a mountain pass
type normalized solution to (\ref{e1.1}) with $\lambda<0$ and
$E(u)=c^{po}$. By Lemma \ref{lem4.7}, $u$ is a normalized ground
state. The proof is complete.

\section{Proof of Theorem \ref{thm1.2}}
\setcounter{section}{4} \setcounter{equation}{0}

Firstly, we study the positivity of the normalized ground states to
(\ref{e1.1}). By Lemma \ref{lem4.7}, it is enough to prove the
following fact.

\begin{proposition}\label{lem4.5}
Assume the conditions in Theorem \ref{thm1.1} hold. If $u \in
\mathcal{P}$  such that $E(u)=c^{po}$, then $|u|\in \mathcal{P}$ and
$E(|u|)=c^{po}$. Moreover, $|u|>0$ in $\mathbb{R}^N$.
\end{proposition}

\begin{proof}
It follows from
\begin{equation*}
\int_{\mathbb{R}^N}|\nabla |u||^2\leq \int_{\mathbb{R}^N}|\nabla
u|^2
\end{equation*}
that $P(|u|)\leq 0$. By Lemma \ref{lem3.6}, there exists
$s_{|u|}\leq 0$ such that $s_{|u|}\star |u|\in \mathcal{P}$. Thus,
\begin{equation*}
\begin{split}
&E(s_{|u|}\star
|u|)\\
&=\left(\frac{\eta_p}{2}-\frac{1}{2p}\right)\int_{\mathbb{R}^N}(I_\alpha\ast|s_{|u|}\star
|u||^p)|s_{|u|}\star |u||^p
+\left(\frac{\gamma_q}{2}-\frac{1}{q}\right)\mu\int_{\mathbb{R}^N}|s_{|u|}\star |u||^q\\
&=\left(\frac{\eta_p}{2}-\frac{1}{2p}\right)e^{(Np-N-\alpha)s_{|u|}}\int_{\mathbb{R}^N}(I_\alpha\ast|u|^p)|u|^p+
\left(\frac{\gamma_q}{2}-\frac{1}{q}\right)\mu e^{(\frac{N}{2}q-N)s_{|u|}}\int_{\mathbb{R}^N}|u|^q\\
&\leq
\left(\frac{\eta_p}{2}-\frac{1}{2p}\right)\int_{\mathbb{R}^N}(I_\alpha\ast|u|^p)|u|^p+
\left(\frac{\gamma_q}{2}-\frac{1}{q}\right)\mu\int_{\mathbb{R}^N}|u|^q\\
&=E(u)=c^{po}.
\end{split}
\end{equation*}
By the definition of $c^{po}$, we have $s_{|u|}=0$, $|u|\in
\mathcal{P}$ and $E(|u|)=c^{po}$.

By Lemma \ref{lem3.3}, there exists $\lambda<0$ such that  $|u|$
satisfies the equation
\begin{equation*}
-\Delta u=\lambda u+(I_\alpha\ast |u|^p)|u|^{p-2}u+\mu|u|^{q-2}u.
\end{equation*}
Since $|u|$ is continuous by Theorem 2.1 in \cite{Li-Ma 2020}, the
strong maximum principle implies  that $|u|>0$ in $\mathbb{R}^N$.
\end{proof}

Nextly, we study the radial symmetry of the normalized ground states
to (\ref{e1.1}). We follow the strategy of \cite{Moroz-Schaftingen
2015}. The argument relies on polarizations. So we first recall some
theories of polarizations (\cite{{Brock-Solynin
2000},{Moroz-Schaftingen JFA 2013},{Van Schaftingen-Willem 2008}}).

Assume that $H\subset \mathbb{R}^N$ is a closed half-space and that
$\sigma_H$ is the reflection with respect to $\partial H$. The
polarization $u^H : \mathbb{R}^N \to \mathbb{R}$ of $u :
\mathbb{R}^N \to \mathbb{R}$ is defined for $x\in \mathbb{R}^N$ by
\begin{equation*}
u^H(x)=\left\{
\begin{array}{ll}
\max\{u(x),u(\sigma_H(x))\}, &\ \mathrm{if}\ x\in H,\\
\min\{u(x),u(\sigma_H(x))\}, &\ \mathrm{if}\ x\not\in H.
\end{array}
\right.
\end{equation*}

We will use the following standard property of polarizations (Lemma
5.3 in \cite{Brock-Solynin 2000}).

\begin{lemma}\label{lem4.1}
(Polarization and Dirichlet integrals). If $u\in H^1(\mathbb{R}^N)$,
then $u^H\in H^1(\mathbb{R}^N)$ and
\begin{equation*}
\int_{\mathbb{R}^N}|\nabla u^H|^2=\int_{\mathbb{R}^N}|\nabla u|^2.
\end{equation*}
\end{lemma}

We shall also use a polarization inequality with equality cases
(Lemma 5.3 in \cite{Moroz-Schaftingen JFA 2013}).

\begin{lemma}\label{lem4.2}
(Polarization and nonlocal integrals). Let $\alpha\in(0,N)$, $u\in
L^{\frac{2N}{N+\alpha}} (\mathbb{R}^N)$ and $H\subset \mathbb{R}^N$
be a closed half-space. If $u\geq 0$, then
\begin{equation*}
\int_{\mathbb{R}^N}\int_{\mathbb{R}^N}\frac{u(x)u(y)}{|x-y|^{N-\alpha}}dxdy
\leq
\int_{\mathbb{R}^N}\int_{\mathbb{R}^N}\frac{u^H(x)u^H(y)}{|x-y|^{N-\alpha}}dxdy,
\end{equation*}
with equality if and only if either $u^H = u$ or $u^H = u\circ
\sigma_H$.
\end{lemma}

The last tool that we need is a characterization of symmetric
functions by polarizations (Proposition 3.15 in \cite{Van
Schaftingen-Willem 2008}, Lemma 5.4 in \cite{Moroz-Schaftingen JFA
2013}).

\begin{lemma}\label{lem4.4}
(Symmetry and polarization). Assume that $u\in L^2(\mathbb{R}^N)$ is
nonnegative. There exist $x_0\in \mathbb{R}^N$ and a non-increasing
function $v : (0,\infty) \to \mathbb{R}$ such that for almost every
$x\in \mathbb{R}^N$, $u(x)=v(|x-x_0|)$ if and only if for every
closed half-space $H\subset \mathbb{R}^N$, $u^H = u$ or $u^H =
u\circ \sigma_H$.
\end{lemma}

Now we are ready to prove the radial symmetry of the positive
normalized ground states to (\ref{e1.1}).

\begin{proposition}\label{lem4.6}
 Assume that the conditions in Theorem
\ref{thm1.1} hold. Let $u$ be a positive normalized ground state to
(\ref{e1.1}), then there exist $x_0\in \mathbb{R}^N$ and a
non-increasing positive function $v : (0,\infty) \to \mathbb{R}$
such that $u(x)=v(|x-x_0|)$ for almost every $x\in \mathbb{R}^N$.
\end{proposition}

\begin{proof}
By Lemma \ref{lem4.7}, $E(u)=c^{po}$ and $P(u)=0$. Let $\tilde{E}$
and $\Gamma$
 be defined in (\ref{e3.3}) and (\ref{e3.5}),
respectively, and let $\gamma_u(\tau)=(0,((1-\tau)s_0+\tau s_1)\star
u)\in \Gamma$ be a path defined in (\ref{e3.6}). Denote
$\beta_u(\tau)=((1-\tau)s_0+\tau s_1)\star u$. Then,
$\beta_u\left({-s_0}/{(s_1-s_0)}\right)=u$, $\beta_u(\tau)\geq 0$
for every $\tau\in[0,1]$,
$\tilde{E}(\gamma_u(\tau))=E(\beta_u(\tau))<E(u)=c^{po}$ for any
$\tau\in \left([0,1]\setminus
\left\{{-s_0}/{(s_1-s_0)}\right\}\right)$.

For every closed half-space $H$ define the path $\gamma_u^H : [0, 1]
\to \mathbb{R}\times S_a$ by $\gamma_u^H(\tau) =
(0,(\beta_u(\tau))^H)$. By Lemma \ref{lem4.1} and
$\|u^H\|_r=\|u\|_r$ with $r\in [1,\infty)$, we have $\gamma_u^H \in
C([0, 1],\mathbb{R}\times S_a)$. By Lemmas \ref{lem4.1} and
\ref{lem4.2}, we obtain that $\tilde{E}(\gamma_u^H(\tau))\leq
\tilde{E}(\gamma_u(\tau))$ for every $\tau\in [0,1]$ and then
$\gamma_u^H(\tau)\in \Gamma$. Hence,
\begin{equation*}
\max_{\tau\in[0,1]}\tilde{E}(\gamma_u^H(\tau))\geq c^{po}.
\end{equation*}
Since for every $\tau\in \left([0,1]\setminus
\left\{{-s_0}/{(s_1-s_0)}\right\}\right)$,
\begin{equation*}
\tilde{E}(\gamma_u^H(\tau))\leq
\tilde{E}(\gamma_u(\tau))<E(u)=c^{po},
\end{equation*}
we deduce that
\begin{equation*}
\tilde{E}\left(\gamma_u^H\left(\frac{-s_0}{s_1-s_0}\right)\right)
=E(u^{H})=c^{po}.
\end{equation*}
Hence $E(u^H)=E(u)$, which implies that
\begin{equation*}
\int_{\mathbb{R}^N}\left(I_\alpha\ast
\left|u^H\right|^p\right)\left|u^H\right|^p=\int_{\mathbb{R}^N}(I_\alpha\ast
|u|^p)|u|^p.
\end{equation*}
By Lemma \ref{lem4.2}, we have $u^H=u$ or $u^H=u\circ \sigma_H$. By
Lemma \ref{lem4.4}, we complete the proof.
\end{proof}

\textbf{Proof of Theorem \ref{thm1.2}}. By Lemma \ref{lem4.7} and
Proposition \ref{lem4.5}, if $u\in \mathcal{G}$, then $|u|\in
\mathcal{G}$ and $\|\nabla |u|\|_2=\|\nabla u\|_2$. Then we can
proceed as Theorem 4.1 in \cite{Hajaiej-Stuart 2004}, obtaining that
\begin{equation*}
\mathcal{G}=\{e^{i\theta}|v|: \ \theta\in \mathbb{R}\ \mathrm{and}\
|v|>0\ \mathrm{in} \ \mathbb{R}^N\}.
\end{equation*}
Moreover, by Proposition \ref{lem4.6}, the proof is complete.

\bigskip

{\bf Acknowledgements.} This work is supported by the National
Natural Science Foundation of China (No. 12001403).



\end{document}